\documentclass[11pt,twoside]{amsart}
\usepackage{amsmath, amsthm, amscd, amsfonts, amssymb, graphicx, color}
\usepackage[bookmarksnumbered, plainpages]{hyperref}
\usepackage{epsfig}
\usepackage{caption}
\usepackage{tikz}
\newtheorem{thm}{Theorem}[section]
\newtheorem{corollary}[thm]{Corollary}
\newtheorem{lem}[thm]{Lemma}
\newtheorem{proposition}[thm]{Proposition}

\numberwithin{equation}{section}

\newtheorem{lemma}[thm]{\bf Lemma}

\newtheorem{definition}[thm]{\bf Definition}

\begin{document}

\title{}
\author{}
\begin{center}
{\bf \Large Signed Roman Domination Number and Join of Graphs}

\vspace*{5mm}
{   Ali Behtoei$^*$, Ebrahim Vatandoost, Fezzeh Azizi Rajol Abad}

{ a.behtoei@sci.ikiu.ac.ir,~~e-vatandoost@ikiu.ac.ir,~~vf.azizi66@gmail.com}

{Department of Mathematics, Imam Khomeini International University,}

{ P.O. Box 34149-16818, Qazvin, Iran}
\end{center}
\thanks{{\scriptsize
\hskip -0.4 true cm MSC(2010): Primary: 05C69 ; Secondary: 05C78.
$*$ Corresponding author}}
\maketitle


\begin{abstract}
A signed Roman dominating function (simply, a ``SRDF") on a graph $G=(V,E)$ is a
function  $f:V(G)\rightarrow \{-1,1,2\}$  satisfying the conditions that (i) the sum of its function
values over any closed neighborhood is at least one and (ii) each vertex $x$ for
which $f(x)=-1$ is adjacent to at least one vertex $y$ for which $f(v)=2$. The weight
of a SRDF is the sum of its function values over all vertices. The signed Roman
domination number of $G$, denoted by $\gamma_{_{sR}}(G)$, is the minimum weight of a SRDF on $G$.
In this paper we study the signed Roman domination number of the join of graphs.
We determine it for the join of cycles,
wheels, fans, and friendship  graphs. 
\end{abstract}

\vskip 0.2 true cm


\pagestyle{myheadings}
\markboth{\rightline {\sl Signed Roman Domination Number and Join of Graphs \hskip 4 cm  A. Behtoei, E. Vatandoost, F. Azizi   }}
         {\leftline{\sl Signed Roman Domination Number and Join of Graphs \hskip 4 cm  A. Behtoei, E. Vatandoost, F. Azizi }}

\bigskip
\bigskip


\section{\bf Introduction}
\vskip 0.4 true cm

Throughout this paper we consider (non trivial) simple graphs, that are finite and undirected graphs without loops or multiple edges.
 Let $G=(V (G),E(G))$ be a connected graph of order $n=|V(G)|$ and of size $m=|E(G)|$.
When $x$ is a vertex of $G$, then the open neighborhood of $x$ in $G$ is the set $N_{_G}(x)=\{y:~xy\in E(G)\}$
and the closed neighborhood of $x$ in $G$ is the set $N_{_G}[x] = N_{_G}(x)\cup\{x\}$.
The degree of vertex $x$  is the number of edges adjacent to $x$ and is denoted by $\deg_G(x)$ .
The minimum degree and the maximum degree of $G$ are denoted by $\Delta(G)$ and $\delta(G)$, respectively. 


A set $D\subseteq V(G)$ is called a {\bf dominating set} of $G$ if each vertex outside $D$ has at least one neighbor in $D$. 
The minimum cardinality of a dominating set of $G$ is the {\bf domination number} of $G$ and is denoted by $\gamma(G)$.
For example, the domination numbers of the $n$-vertex complete graph, path, and cycle are given by $\gamma(K_n)=1$, $\gamma(P_n)=\lceil{n\over3}\rceil$ and $\gamma(C_n)=\lceil{n\over3}\rceil$, respectively \cite{Hynes}.
Domination is a rapidly developing area of research in graph theory, and its
various applications to ad hoc networks, distributed computing, social networks, biological networks
and web graphs partly explain the increased interest. The concept
of domination has existed and studied for a long time and early discussions on the topic
can be found in the works of  Berge \cite{Berge} and Ore \cite{Ore}.
At present, domination is considered to be one of the fundamental concepts in
graph theory with an extensive research activity.
Garey and Johnson \cite{complexity} have shown that determining the domination number of an arbitrary graph is an NP-complete problem.
The domination number can be defined equivalently by means of a function, which can be considered as a characteristic function of a dominating set, see \cite{Hynes}. 
A function $f:V(G)\rightarrow \{0,1\}$ is called a {\bf dominating function} on  $G$ if for
each vertex $x\in V(G)$, $\sum_{y\in N_G[x]}f(y)\geq 1$. 
The value $w(f)=\sum_{x\in V(G)}f(x)$ is called the {\bf weight} of $f$. 
Now, the domination number of $G$ can be defined as $$\gamma(G)=\min\{w(f):~f ~\mbox{is a domination function on}~ G\}.$$

Analogously, a {\bf signed domination function} of $G$ is a labeling of the vertices
of $G$ with $+1$ and $-1$ such that the closed neighborhood of each vertex
contains more $+1$'s than $-1$'s. The {\bf signed domination
number} of $G$ is the minimum value of the sum of vertex labels, taken over all
signed domination functions of $G$. This concept is closely related to combinatorial
discrepancy theory as shown by F\"{u}redi and Mubayi in \cite{Furedi}.
In general, many domination parameters are defined by combining domination
with other graph theoretical properties.

\begin{definition}  \label{SRDF}
\cite{Ahangar}  
Let $G=(V,E)$ be a graph. A {\bf signed Roman domination function} (simply, a ``SRDF") on the graph $G$ is a function $f:V\rightarrow \{-1,1,2\}$ which satisfies two following conditions:
\begin{itemize}
 \item[(a)] For each $x\in V$, $\sum_{y\in N_G[x]}f(y)\geq 1$,
 \item[(b)] Each vertex $x$ for which $f(x)=-1$ is adjacent to at least one vertex $y$ for which $f(y)=2$. 
\end{itemize}
The value $f(V)=\sum_{x\in V}f(x)$ is called the {\bf weight} of the function $f$ and is denoted by $w(f)$.
The {\bf signed Roman domination number} of $G$, $\gamma_{_{sR}}(G)$, is the minimum weight of a SRDF on $G$.
\end{definition}

These concepts are introduced by  Ahangar et al. in \cite{Ahangar}.
They described the usefulness of these concepts in various applicative  areas
like ``defending the Roman empire"  (see \cite{Ahangar},  \cite{Hening} and \cite{Stewart} for more details).
It is obvious that for every graph $G$ of order $n$ we have $\gamma_{_{sR}}(G) \leq  n$, because assigning $+1$ to each vertex yields a SRDF.
In \cite{Ahangar} Ahangar et al.  present various
lower and upper bounds on the signed Roman domination number of a graph in terms of it's order, size and vertex degrees. Moreover, they characterized all graphs which attain these bounds. Also, they  investigate the relation between $\gamma_{_{sR}}$ and some other graphical parameters, and the signed Roman domination number of some special bipartite graphs. It is proved in \cite{Ahangar} that $\gamma_{_{sR}}(K_n)=1$ for each $n\neq 3$, $\gamma_{_{sR}}(K_3)=2$, $\gamma_{_{sR}}(C_n)=\lceil{2n\over3}\rceil$, $\gamma_{_{sR}}(P_n)=\lfloor{2n\over3}\rfloor$, and that the only $n$-vertex graph $G$ with $\gamma_{_{sR}}(G)=n$ is the empty graph $\overline{K}_n$.


Note that each signed Roman domination function $f$ of $G$ is uniquely determined by the ordered partition $(V_{-1},V_1,V_2)$ of $V(G)$, where $V_i=\{x\in V(G):~f(x)=i\}$ for each $i\in\{-1,1,2\}$. Specially, $w(f)=2|V_2|+|V_1|-|V_{-1}|$.
For convenience, we usually write $f=(V_{-1},V_1,V_2)$ and, when $S\subseteq V$ we denote the summation $\sum_{x\in S}f(x)$ by $f(S)$.
If $w(f)=\gamma_{_{sR}}(G)$, then $f$ is called a  {\bf $\gamma_{sR}(G)$-function} or an {\bf optimal SRDF} on $G$.
Recall that the join of two graphs $G_1$ and $G_2$,
denoted by $G_1\vee G_2$, is a graph with vertex set $V(G_1)\cup V(G_2)$ and edge set $E(G_1)\cup E(G_2)\cup\{uv:$ $u\in V(G_1),~v\in V(G_2)\}$. For example $K_1\vee P_n$ is the fan $F_n$, $K_1\vee C_n$ is the wheel $W_n$, and the friendship graph $Fr_n$, $n=2m+1$, is the graph obtained by joining $K_1$ to the $m$ disjoint copies of $K_2$.

In this paper we study the signed Roman domination number of the join of graphs.
Specially, we determine the signed Roman domination number of $C_m\vee C_n$,
$W_n$, $F_n$, and friendship  graphs $Fr_n$. 

\section{\bf Join of graphs}
\vskip 0.4 true cm

For investigating $\gamma_{_{sR}}$ of the join of graphs, 
the following  lemma is useful.

\begin{lemma} \label{MaxDeg}
If $G$ is a graph with $\Delta(G)=|V(G)|-1$, then $\gamma_{_{sR}}(G)\geq 1$.
\end{lemma}
\begin{proof}{
Let $f$ be an optimal signed Roman domination function on $G$ and let 
$x\in V(G)$ be a vertex of maximum degree $\Delta(G)$. 
Since $N_G(x)=V(G)\setminus \{x\}$ and using the definition of a $SRDF$, we get
$$\gamma_{_{sR}}(G)=w(f)=\sum_{v\in V(G)}f(v)=f(x)+\sum_{v\in N_G(x)}f(v)=f(N_G[x])\geq 1.$$
}\end{proof}

\begin{corollary}     \label{GvK1}
For each graph $G$,  $\gamma_{_{sR}}(G\vee K_1) \geq1 $. Specially, if $\gamma_{_{sR}}(G)=0$, then $\gamma_{_{sR}}(G\vee K_1) =1$.
\end{corollary}
\begin{proof}{
The first statement follows directly from Lemma \ref{MaxDeg}. Let $f$ be a $\gamma_{sR}(G)$-function of $G$. 
Define $g:V(G\vee K_1) \rightarrow \{ -1 , 1 , 2 \}$ as $g(x)=f(x)$ when $x\in V(G)$, and $g(x)=1$ when $x\in V(K_1)$.
Since $g$ is a SRDF of weight $1$ on $G\vee K_1$, $\gamma_{_{sR}}(G\vee K_1)\leq1$.
Now Corollary \ref{GvK1} implies that $\gamma_{_{sR}}(G\vee K_1)=1$.
}\end{proof}

\begin{proposition} 
Let  $G$ and $H$ be two graphs such that $\gamma_{sR}(G) \geq 0$ and $\gamma_{sR}(H) \geq 0$. Then, 
$$\gamma_{_{sR}}(G\vee H)\leq \gamma_{_{sR}}(G)+\gamma_{_{sR}}(H).$$
\end{proposition}
\begin{proof}{
Let $f_1$ be a $\gamma_{_{sR}}(G)$-function on $G$ and let $f_2$ be a $\gamma_{_{sR}}(H)$-function on $H$. 
Define $f:V(G\vee H) \rightarrow \{-1,1,2\}$ as
$f(x)=f_1(x)$ when $x\in V(G)$, and $f(x)=f_2(x)$ when $x\in V(H)$.
For each $v\in V(G)$,   $f(N_{G\vee H}[v]) =f(N_G[v])+ w(f_1) \geq1$. Similarly, for each $v\in V(H)$,   $f(N_{G\vee H}[v]) =f(N_H[v])+ w(f_2) \geq1$.
Thus, $f$ is a $SRDF$ on $G\vee H$ and
$\gamma_{_{sR}}(G\vee H) \leq w(f)=w(f_1)+w(f_2)=\gamma_{_{sR}}(G)+\gamma_{_{sR}}(H)$.
}\end{proof}
\section{\bf Join of cycles}
\vskip 0.4 true cm

Since $\Delta(C_m\vee C_n)=\max\{m+2,n+2\}$,  the maximum dergree of $C_m\vee C_n$ is $m+n-1$ if and only if $3\in\{m,n\}$.
Hence, for $m\geq4$ and $n\geq4$ the graph $C_m\vee C_n$ has no vertex of degree $|V(C_m\vee C_n)|-1$.
\begin{proposition}
If $n$ is a multiple of 3, Then $\gamma_{_{sR}}(C_{3} \vee C_{n})=1$.
\end{proposition}
\begin{proof}{
Let $V(C_3)=\{x_1,x_2,x_3\}$ and $V(C_n)=\{y_1,y_2,...,y_n\}$ which are arranged consecutively on a circle, respectively.
Define  $f:V(C_3\vee C_n)\rightarrow \{-1,1,2\}$  as $f(x_1)=f(x_2)=1$, $f(x_3)=-1$ and
$f(y_j)=2$ when  $i\equiv 1~(mod~3)$, and $f(y_j)=-1$ otherwise.   
Note that $f(V(C_3))=1$ and $f(V(C_n))=0$. It is easy to check that $f$ is a SRDF (of weight 1) on $C_3\vee C_n$. Now Lemma \ref{MaxDeg} completes the proof.
}\end{proof}

The following theorem considers more general cases.
\begin{thm} \label{CmCn4}
For each pair of positive integers $m\geq3$ and $n\geq3$, we have $1\leq \gamma_{_{sR}}(C_m\vee C_n)\leq 4$.
\end{thm}
\begin{proof}{
Assume that $V(C_m)=\{x_1,x_2,...,x_m\}$ and $V(C_n)=\{y_1,y_2,...,y_n\}$ which are arranged consecutively on a circle, respectively.
Without loss of generality, assume that $m$ is odd and $n$ is even (other cases are similar).
Define two functions $f_o:V(C_m)\rightarrow \{-1,1,2\}$ and  $f_e:V(C_n)\rightarrow \{-1,1,2\}$ as
\begin{eqnarray*}
f_o(x_i)=\left\{ \begin{array}{ll}     2 & i=1 \\ -1 & i\in\{ 2,4,...,n-1\} \\ 1 & i\in\{3,5,...,n\},   \end{array} \right.  
\hspace*{10mm}
f_e(y_i)=\left\{ \begin{array}{ll}     2 & j\in\{1,3\} \\ -1 & j\in\{2,4,...,n\} \\ 1 & j\in\{5,7,...,n-1\}.  \end{array} \right.
\end{eqnarray*}
Now define $f:V(C_m\vee C_n)\rightarrow \{-1,1,2\}$ as $f(v)=f_o(x_i)$ when  $v=x_i$, and  $f(v)=f_e(y_j)$ when $v=y_j$. 
Note that $f(x_1)=f(y_1)=2$ and each vertex  in $C_m\vee C_n$ is adjacent to  $x_1$ or $y_1$.
Also, $f(V(C_m))=f(V(C_n))=2$ and for each $i,j$ we have $f_o(N_{C_m}[x_i])\geq -1$ and $f_e(N_{C_n}[y_j])\geq -1$.
Hence, $$f(N_{C_m\vee C_n}[x_i])=f_o(N_{C_m}[x_i])+f_e(V(C_n))\geq -1+ 2=1$$
and $$f(N_{C_m\vee C_n}[y_j])=f_e(N_{C_n}[y_j])+f_o(V(C_m))\geq -1+ 2=1.$$
Thus, $f$ is a SRDF on $C_m\vee C_n$ and $w(f)=f_o(C_m)+f_e(C_n)=2+2=4$, the upper bound follows.
\\
In order to obtain the lower bound, let $g$ be an optimal SRDF on $C_m\vee C_n$. If $g(V(C_m))\geq 1$ and $g(V(C_m))\geq 1$, then the result follows.
Assume that $g(V(C_n))=\alpha \leq 0$. Since $g$ is a SRDF, for each $x \in V(C_m)$ we have $g(N_{C_m\vee C_n} [x]) \geq1$. 
This using the fact $g(N_{C_m\vee C_n} [x])=g(N_{C_m} [x])+g(V(C_n))$ implies that $g(N_{ C_m} [x]) \geq 1-\alpha$.
Hence,  
\begin{eqnarray*}
g(V(C_m)) =\sum_{x \in V(C_m)} g(x) = {1\over 3} \sum_{x \in V(C_m)} g(N_{ C_m} [x]) 
\geq {1\over 3} \sum_{x \in V(C_m)} (1-\alpha)  \geq {m\over 3} (1-\alpha).
\end{eqnarray*}
This implies that
$$\gamma_{_{sR}}(C_m\vee C_n)=w(g)=g(V(C_m))+g(V(C_n))\geq {m\over 3} (1-\alpha)+\alpha = {m\over 3} + ({m\over 3}-1)(-\alpha) \geq 1.$$
A similar argument holds for the situation $g(V(C_n)) \leq 0$. This completes the proof.
}\end{proof}

\begin{lem} \label{atLeast1}
Let $m\geq 13$ and $n\geq 13$ be two positive integers.  If $f$ is a SRDF on $C_m\vee C_n$, 
then $f(V(C_m))> 0$ and $f(V(C_n))>0$. Specially,  $\gamma_{_{sR}}(C_m\vee C_n) \geq 2$.
\end{lem}
\begin{proof}{
Suppose on the contrary that $f$ is a SRDF on $C_m\vee C_n$ and $f(C_n)=\alpha \leq 0$. Since $f$ is a SRDF, for each $x \in V(C_m)$ we have $f(N_{C_m\vee C_n} [x]) \geq1$. 
This using the fact $f(N_{C_m\vee C_n} [x])=f(N_{C_m} [x])+f(V(C_n))$ implies that $f(N_{ C_m} [x]) \geq |\alpha|+1$.
Hence,  
\begin{eqnarray*}
f(V(C_m)) = {1\over 3} \sum_{x \in V(C_m)} f(N_{ C_m} [x]) 
\geq {1\over 3} \sum_{x \in V(C_m)} (|\alpha|+1)  \geq {1\over 3} m(|\alpha|+1).
\end{eqnarray*}
Therefore,
\begin{eqnarray*}
w(f)=f(V(C_m)) + f(V(C_n)) \geq {m\over 3} (|\alpha|+1)+\alpha  \geq {13\over 3} (-\alpha+1)+\alpha ={-10\alpha\over 3}+{13\over 3} >4.
\end{eqnarray*}
This contradicts Theorem \ref{CmCn4}. Thus, $f(V(C_n))\geq 1$. Similarly, we can prove that  $f(V(C_m))\geq 1$.
}\end{proof}

The following corollary is an immidiate consequence of the proof of Lemma \ref{atLeast1}.
\begin{corollary}
Let $m\geq 13$ and $n\geq 13$ be two positive integers.  If $f$ is an optimal  SRDF on $C_m\vee C_n$ 
such that $f(N_{C_m}(x))<0$ for some $x\in V(C_m)$, then  $\gamma_{_{sR}}(C_m\vee C_n) \geq 3$.
\end{corollary}

\begin{thm}  \label{mn3k+2}
Let $m\geq13$ and $n\geq13$ be two positive integers. If $m\equiv2~(mod~3)$ and $n\equiv2~(mod~3)$, then $\gamma_{_{sR}}(C_m\vee C_n)=2$.  
\end{thm}
\begin{proof}{
Define the function $f$ from $V(C_m)\cup V(C_n)=\{x_1,...,x_m\}\cup\{y_1,...,y_n\}$ to $\{-1,1,2\}$ as follows.
\begin{eqnarray*}
f(x_i)=\left\{ \begin{array}{ll}   2 & i\equiv 1~(mod~3)  \\ -1 & o.w.  \end{array}\right.
\hspace*{10mm}
f(y_j)=\left\{ \begin{array}{ll}   2 & j\equiv 1~(mod~3)  \\ -1 & o.w.  \end{array}\right.
\end{eqnarray*}
Hence, $f(V(C_m))=f(V(C_n))=1$,  $f(N_{C_m}[x_m])=f(N_{C_n}[y_n])=3$ and for each $1\leq i<m$ and each $1\leq j<n$ we have $f(N_{C_m}[x_i])=f(N_{C_n}[y_j])=0$. Thus, $f$ is a SRDF of weight 2.
Now Lemma \ref{atLeast1} completes the proof.
}\end{proof}

\begin{lemma}  \label{NegativeNeighbour}
Let $n\geq13$ be an integer such that  $n\not\equiv2\pmod{3}$. 
If $f:V(C_n)\rightarrow \{-1,1,2\}$ is a function for which $f(V(C_n))=1$, then there exists $y\in V(C_n)$ such that $f(N_{C_n}[y])<0$.
\end{lemma}
\begin{proof}{
Since $1=f(V(C_n))={1\over3}\sum_{x\in V(C_n)}f(N_{C_n}[x])$,
the summation $\sum_{x\in V(C_n)}f(N_{C_n}[x])$ is  equal to $3$.
Assume on the contrary that $f(N_{C_n}[y])\geq 0$ for each $y\in V(C_n)$. Thus, one of the following cases should be happened.
\begin{itemize}
\item[i)]  There exists $y\in V(C_n)$ such that $f(N_{C_n}[y])=3$ and $f(N_{C_n}[y'])=0$ for each $y'\neq y$.
\item[ii)]  There exist $y,y'\in V(C_n)$ such that $f(N_{C_n}[y])=2$, $f(N_{C_n}[y'])=1$ and $f(N_{C_n}[y''])=0$ for each $y''\notin\{y,y'\}$.
\item[iii)]  There exist $y,y',y''\in V(C_n)$ such that $f(N_{C_n}[y])=f(N_{C_n}[y'])=f(N_{C_n}[y''])=1$ and $f(N_{C_n}[\bar{y}])=0$ for each $\bar{y}\notin\{y,y',y''\}$.
\end{itemize}
{\bf Claim.} There exists no vertex with label $1$.

In order to prove this claim, suppose (on the contrary) that $f(y_j)=1$ for some $y_j\in V(C_n)=\{y_1,y_2,...,y_n\}$. 
We consider the following possibilities for the labels of the neighbours of $y_j$.
\begin{itemize}
\item[1)]  $f(y_{j-1})=1$ and $f(y_{j+1})=1$:\\
 This implies that $f(N_{C_n}[y_j])=3$ and $f(N_{C_n}[y_{j-1}])\geq 1$, which contradicts the above three possible cases (i), (ii) and (iii).
\item[2)]  $f(y_{j-1})=2$ and $f(y_{j+1})=2$:\\
 This implies that $f(N_{C_n}[y_j])=5$, which is a contradiction.
\item[3)]  $f(y_{j-1})=2$ and $f(y_{j+1})=1$:\\
 Hence $f(N_{C_n}[y_j])=4$, which is a contradiction.
\item[4)]  $f(y_{j-1})=2$ and $f(y_{j+1})=-1$:\\
 This implies that $f(N_{C_n}[y_j])=2$ and $f(N_{C_n}[y_{j-1}])\geq2$, which is a contradiction.
\item[5)]  $f(y_{j-1})=-1$ and $f(y_{j+1})=-1$:\\
 Thus $f(N_{C_n}[y_j])=-1$, which is a contradiction.
 \item[6)]  $f(y_{j-1})=1$ and $f(y_{j+1})=-1$:\\
Since $f(N_{C_n}[y_{j+1}])\geq 0$, $f(y_{j+2})\in\{1,2\}$. Since $f(N_{C_n}[y_{j}])=1$, $f(N_{C_n}[y_{j-1}])\geq 1$ and  $f(N_{C_n}[y_{j+1}])\geq 1$, we should have $f(N_{C_n}[y_{j+1}])=1$ and $f(y_{j+2})=1$. Therefore, $f(N_{C_n}[y_{j'}])=0$ for each $j'\notin\{j-1,j,j+1\}$ and specially $f(N_{C_n}[y_{j+2}])=0$, which is impossible.
\end{itemize}
This completes the proof of the claim. 
Therefore, the label of each vertex in $C_n$ is $-1$ or $2$. Let $t$ be the number of vertices whose label is $2$.
If $n=3k$, then
$1=f(V(C_n))=2t+(3k-t)(-1)=3(t-k)$,
which is a contradiction (3 is not a divisor of 1).
If $n=3k+1$, then
$1=2t+(3k+1-t)(-1)$ and hence, $2=3(t-k)$ which is a contradiction. 
}\end{proof}

\begin{thm}   \label{n3k3k+1}
Let $m\geq13$ and $n\geq13$ be two integers such that $m\equiv 2\pmod{3}$ and $n\not\equiv2\pmod{3}$. Then $\gamma_{_{sR}}(C_m \vee C_n)=3$.
\end{thm}
\begin{proof}{
Define the function $g$ on $V(C_m)=\{x_1,...,x_m\}$ as 
$g(x_i)=2$  when $i\equiv 1\pmod{3}$, and $g(x_i)=-1$ otherwise.
Thus,  $g(N_{C_m}[x_m])=3$, $g(N_{C_m}[x_i])=0$ for each $i\neq m$, and $g(V(C_m))=1$. 
When $n\equiv0\pmod{3}$ (or $n\equiv1\pmod{3}$) define the function $h_1$ (or $h_2$) on $V(C_n)=\{y_1,...,y_n\}$ as follows.
\begin{eqnarray*}
h_1(y_j)=\left\{ \begin{array}{ll}   1 & j=n \\ 2 & j\equiv1\pmod{3}  \\ -1 & o.w.  \end{array}\right.
\hspace*{2mm}
,
\hspace*{9mm}
h_2(y_j)=\left\{ \begin{array}{ll}   2 & j\equiv 1\pmod{3}  \\ -1 & o.w.  \end{array}\right.
\end{eqnarray*}
Note that $h_1(V(C_n))=2$ and $h_1(N_{C_n}[y_j])\geq 0$ for each $j$ 
(similarly, $h_2(V(C_n))=2$ and $h_2(N_{C_n}[y_j])\geq 0$ for each $j$). 
Now $g$ using  $h_1$ (or $h_2$) induces a labelling on $V(C_m\vee C_n)$ which is a SRDF of weight 1+2=3.
Hence,  $\gamma_{_{sR}}(C_m \vee C_n)\leq 3$.
Let $f$ be an optimal SRDF on $C_m \vee C_n$.
By Lemma \ref{atLeast1},   $f(V(C_m))\geq 1$ and $f(V(C_n))\geq 1$. If $f(V(C_n))\geq 2$, then we are done. 
Else $f(V(C_n))=1$ and Lemma \ref{NegativeNeighbour} implies that there exists $y\in V(C_n)$ such that $f(N_{C_n}[y])\leq -1$.
Since $f(N_{C_m\vee C_n}[y])\geq1$, we should have $f(V(C_m))\geq 2$. Thus, $w(f)=f(V(C_m))+f(V(C_n))\geq 3$, which completes the proof.
}\end{proof}

\begin{thm}
Let $m\geq13$ and $n\geq13$ be two integers such that $m\not\equiv 2\pmod{3}$ and $n\not\equiv2\pmod{3}$. Then $\gamma_{_{sR}}(C_m \vee C_n)=3$.
\end{thm}
\begin{proof}{
Let $f$ be an optimal SRDF on $C_m \vee C_n$. By Lemma \ref{atLeast1}, $f(V(C_m))\geq 1$ and $f(V(C_n))\geq 1$. 
Lemma \ref{NegativeNeighbour} implies that the case $f(V(C_m))=f(V(C_n))=1$ is impossible.
Thus  $\gamma_{_{sR}}(C_m \vee C_n)\geq 3$.
Using $h_1$ or $h_2$ from the proof of Theorem \ref{n3k3k+1} we obtain a labeling on $V(C_n)$ with total weight 2. 
For the case $m\equiv0\pmod{3}$ (or $m\equiv1\pmod{3}$) define the function $g_1$ (or $g_2$) on $V(C_m)$ as follows.
\begin{eqnarray*}
g_1(x_i)=\left\{ \begin{array}{ll}   1 & i\in\{m-2,m-1\} \\ 2 & i\neq m-2,~i\equiv 1~(mod~3)  \\ -1 & o.w.  \end{array}\right.
\hspace*{3mm}
,
\hspace*{9mm}
g_2(x_i)=\left\{ \begin{array}{ll}   1 & i=m \\ 2 & i\neq m,~i\equiv 1~(mod~3)  \\ -1 & o.w.  \end{array}\right.
\end{eqnarray*}
Note that $g_k(V(C_m))=1$ and  for each $1\leq i\leq m$ we have $g_k(N_{C_m}(x_i))\geq -1$, $k\in\{1,2\}$. 
Now regards to the possible cases for $m$ and $n$, and using one of two functions $g_1,g_2$ and one of two functions $h_1,h_2$ we obtain a labelling on $V(C_m)\cup V(C_n)$  which induces a SRDF of weight 3 on $C_m\vee C_n$.
}\end{proof}


\section{\bf Wheels, Fans and Friendship graphs}
\vskip 0.4 true cm

The following theorem shows that  signed Roman domination number of a wheel  almost always is 1.

\begin{thm} \label{theorem:wheel}
Let $W_n=K_1\vee C_n$ be a wheel of order $n+1$. Then, $\gamma_{_{sR}}(W_4)=2$ and $\gamma_{_{sR}}(W_n)=1$ for each $n\neq 4$.
\end{thm}
\begin{proof}{
Let $V(W_n)=\{v_0,v_1,v_2,...,v_n\}$ and $E(W_n)=\{v_0v_i:~1\leq i\leq n \}\cup \{v_1v_2,v_2v_3,...,v_{n-1}v_n,v_nv_1\}$. Since $\Delta(W_n)=|V(W_n)|-1$, Lemma \ref{MaxDeg} implies that $\gamma_{_{sR}}(W_n)\geq 1$.
For the case $n=4$ it is not hard to check by inspection that there exists no signed Roman domination function on $W_4$ of weight 1 while, Figure \ref{fig:w4w5} (a) illustrates an $SRDF$ on $W_4$ of weight 2. Hence $\gamma_{_{sR}}(W_4)=2$.
To complete the proof it is sufficient to provide a signed Roman domination function of weight 1 on $W_n$ for each $n\neq 4$. For this reason we consider the following different cases.
\\ 
{\bf Case 1.} $n$ is odd:

Define the function $f:V(W_n)\rightarrow \{-1,1,2\}$ as below.
Figure \ref{fig:w4w5} (b) illustrate it for the case  $n=5$ where, the central vertex is $v_0$, top one is $v_1$ and $v_2$ is the second vertex when the sense of traversal being clockwise.
\begin{eqnarray} \label{case1}
f(v_i)= \left\{    \begin{array}{ll}
2 & i=0 \\ -1 & i\equiv 1 ~(\mbox{mod} ~2) \\ 1 & o.w.
\end{array} \right.
\end{eqnarray}
Note that $f$ is a $SRDF$ on $W_n$ of weight $w(f)=f(N_{_{W_n}}[v_0])=1$.
\\ {\bf Case 2.} $n$ is even and $n\equiv 0~(\mbox{mod}~3)$:

Define the function $f:V(W_n)\rightarrow \{-1,1,2\}$ as below. Figure \ref{fig:w12} (a) depicts it for the case $n=12$.
\begin{eqnarray} \label{case2}
f(v_i)= \left\{    \begin{array}{ll}
1 & i=0 \\ 2 & i\geq1,~i\equiv 0 ~(\mbox{mod} ~3) \\ -1 & o.w.
\end{array} \right.
\end{eqnarray}
It is straightforward  to check that
 $f$ is a $SRDF$ on $W_n$ of weight $1$.
\\ {\bf Case 3.} $n$ is even and $n\equiv 1~(\mbox{mod}~3)$.

Define the function $f$ on $V(W_n)$ as follows. Figure \ref{fig:w12} (b) illustrates it for the case $n=10$.
\begin{eqnarray} \label{case3}
f(v_i)= \left\{    \begin{array}{ll}
2& i=0 \\ 2 & 1\leq i\leq n-7,~i\equiv 0 ~(\mbox{mod} ~3) \\  1 & i\in\{n-4,n-1,n\} \\ -1 & o.w.
\end{array} \right.
\end{eqnarray}
It is not hard to check that
 $f$ is a $SRDF$ on $W_n$ and $w(f)=1$.
\\ {\bf Case 4.} $n$ is even and $n\equiv 2~(\mbox{mod}~3)$.

Define the function $f$ on $V(W_n)$ as follow. Figure \ref{fig:w12} (b) depicts it for the case $n=8$.
\begin{eqnarray} \label{case4}
f(v_i)= \left\{    \begin{array}{ll}
2 & i=0 \\ 2 & 1\leq i\leq n-5,~i\equiv 0 ~(\mbox{mod} ~3) \\  1 & i\in\{n-2,n\} \\ -1 & o.w.
\end{array} \right.
\end{eqnarray}
It is easy to check that
 $f$ is a $SRDF$ on $W_n$ and it's weight is one.\\
Therefore, in each case we provide a SRDF on $W_n$ of weight one. This completes the proof.
}\end{proof}
\begin{center}
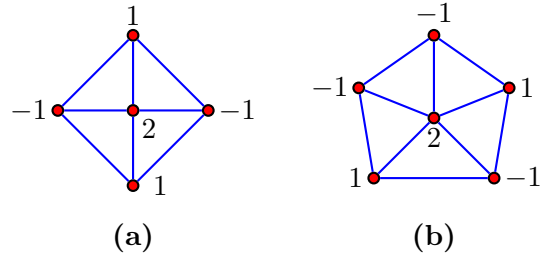

\begin{tikzpicture}
[inner sep=0.5mm, place/.style={circle,draw=black,fill=red,thick}]
\node[place] (v_1) at (-2,1) [label=above:$1$] {};
\node[place] (v_2) at (-1,0) [label=right:$-1$] {}edge [-,thick,blue](v_1);
\node[place] (v_3) at (-2,-1) [label=right:$~1$] {}edge [-,thick,blue](v_2);
\node[place] (v_4) at (-3,0) [label=left:$-1$] {}edge [-,thick,blue](v_3) edge [-,thick,blue](v_1);
\node[place] (v_0) at (-2,0) [label=below right:$2$] {}edge [-,thick,blue](v_1) edge [-,thick,blue](v_2) edge [-,thick,blue](v_3) edge [-,thick,blue](v_4);
\node (dots) at (-2,-1.7) [label=center:{\bf (a)}]{};
\node[place] (v_1) at (2,1) [label=above:$-1$] {};
\node[place] (v_2) at (3,.3) [label=right:$1$] {}edge [-,thick,blue](v_1);
\node[place] (v_3) at (2.8,-.9) [label=right:$-1$] {}edge [-,thick,blue](v_2);
\node[place] (v_4) at (1.2,-.9) [label=left:$1$] {}edge [-,thick,blue](v_3);
\node[place] (v_5) at (1,.3) [label=left:$-1$] {}edge [-,thick,blue](v_4) edge [-,thick,blue](v_1);
\node[place] (v_0) at (2,-.1) [label=below:$2$] {}edge [-,thick,blue](v_1) edge [-,thick,blue](v_2) edge [-,thick,blue](v_3) edge [-,thick,blue](v_4)  edge [-,thick,blue](v_5);
\node (dots) at (2,-1.7) [label=center:{\bf (b)}]{};
\end{tikzpicture}
\captionof{figure}{Signed Roman domination labeling on  $W_4$ and $W_5$.\label{fig:w4w5}}
\end{center}
\begin{center}
\begin{tikzpicture}
[inner sep=0.5mm, place/.style={circle,draw=black,fill=red,thick}]
\node[place] (v_0) at (-4.3,0) {};      
\draw[blue,thick] (v_0) -- +(0:1.5cm);
\draw[blue,thick] (v_0) -- +(30:1.5cm);
\draw[blue,thick] (v_0) -- +(60:1.5cm);
\draw[blue,thick] (v_0) -- +(90:1.5cm);
\draw[blue,thick] (v_0) -- +(120:1.5cm);
\draw[blue,thick] (v_0) -- +(150:1.5cm);
\draw[blue,thick] (v_0) -- +(180:1.5cm);
\draw[blue,thick] (v_0) -- +(210:1.5cm);
\draw[blue,thick] (v_0) -- +(240:1.5cm);
\draw[blue,thick] (v_0) -- +(270:1.5cm);
\draw[blue,thick] (v_0) -- +(300:1.5cm);
\draw[blue,thick] (v_0) -- +(330:1.5cm);
\filldraw[red,draw=black] (v_0)  ++(0:1.5cm) circle (2.2pt);
\filldraw[red,draw=black] (v_0)  ++(30:1.5cm) circle (2.2pt);
\filldraw[red,draw=black] (v_0)  ++(60:1.5cm) circle (2.2pt);
\filldraw[red,draw=black] (v_0)  ++(90:1.5cm) circle (2.2pt);
\filldraw[red,draw=black] (v_0)  ++(120:1.5cm) circle (2.2pt);
\filldraw[red,draw=black] (v_0)  ++(150:1.5cm) circle (2.2pt);
\filldraw[red,draw=black] (v_0)  ++(180:1.5cm) circle (2.2pt);
\filldraw[red,draw=black] (v_0)  ++(210:1.5cm) circle (2.2pt);
\filldraw[red,draw=black] (v_0)  ++(240:1.5cm) circle (2.2pt);
\filldraw[red,draw=black] (v_0)  ++(270:1.5cm) circle (2.2pt);
\filldraw[red,draw=black] (v_0)  ++(300:1.5cm) circle (2.2pt);
\filldraw[red,draw=black] (v_0)  ++(330:1.5cm) circle (2.2pt);
\filldraw[red,draw=black] (v_0)  ++(360:1.5cm) circle (2.2pt);
\draw[blue,thick] (v_0) circle (1.5cm);
\node (dots) at (-4.3,1.8) [label=center:$-1$]{};
\node (dots) at (-3.4,1.6) [label=center:$-1$]{};
\node (dots) at (-2.8,1) [label=center:$2$]{};
\node (dots) at (-2.4,0) [label=center:$-1$]{};
\node (dots) at (-2.7,-.9) [label=center:$-1$]{};
\node (dots) at (-3.3,-1.5) [label=center:$2$]{};
\node (dots) at (-4.3,-1.8) [label=center:$-1$]{};
\node (dots) at (-5.2,-1.6) [label=center:$-1$]{};
\node (dots) at (-5.9,-.85) [label=center:$2$]{};
\node (dots) at (-6.2,0) [label=center:$-1$]{};
\node (dots) at (-6,0.85) [label=center:$-1$]{};
\node (dots) at (-5.2,1.6) [label=center:$2$]{};
\node (dots) at (-3.6,-0.2) [label=center:$1$]{};
\node (dots) at (-4.3,-2.7) [label=center:{\bf (a)}]{};
\node[place] (v_0) at (0,0) {};      
\draw[blue,thick] (v_0) -- +(0:1.5cm);
\draw[blue,thick] (v_0) -- +(30:1.5cm);
\draw[blue,thick] (v_0) -- +(60:1.5cm);
\draw[blue,thick] (v_0) -- +(90:1.5cm);
\draw[blue,thick] (v_0) -- +(150:1.5cm);
\draw[blue,thick] (v_0) -- +(210:1.5cm);
\draw[blue,thick] (v_0) -- +(240:1.5cm);
\draw[blue,thick] (v_0) -- +(270:1.5cm);
\draw[blue,thick] (v_0) -- +(300:1.5cm);
\draw[blue,thick] (v_0) -- +(330:1.5cm);
\filldraw[red,draw=black] (v_0)  ++(0:1.5cm) circle (2.2pt);
\filldraw[red,draw=black] (v_0)  ++(30:1.5cm) circle (2.2pt);
\filldraw[red,draw=black] (v_0)  ++(60:1.5cm) circle (2.2pt);
\filldraw[red,draw=black] (v_0)  ++(90:1.5cm) circle (2.2pt);
\filldraw[red,draw=black] (v_0)  ++(150:1.5cm) circle (2.2pt);
\filldraw[red,draw=black] (v_0)  ++(210:1.5cm) circle (2.2pt);
\filldraw[red,draw=black] (v_0)  ++(240:1.5cm) circle (2.2pt);
\filldraw[red,draw=black] (v_0)  ++(270:1.5cm) circle (2.2pt);
\filldraw[red,draw=black] (v_0)  ++(300:1.5cm) circle (2.2pt);
\filldraw[red,draw=black] (v_0)  ++(330:1.5cm) circle (2.2pt);
\filldraw[red,draw=black] (v_0)  ++(360:1.5cm) circle (2.2pt);
\draw[blue,thick] (v_0) circle (1.5cm);
\node (dots) at (0,1.8) [label=center:$-1$]{};
\node (dots) at (.9,1.6) [label=center:$-1$]{};
\node (dots) at (1.5,1) [label=center:$2$]{};
\node (dots) at (1.9,0) [label=center:$-1$]{};
\node (dots) at (1.6,-.9) [label=center:$-1$]{};
\node (dots) at (.9,-1.6) [label=center:$1$]{};
\node (dots) at (-.3,-1.8) [label=center:$-1$]{};
\node (dots) at (-1.1,-1.5) [label=center:$-1$]{};
\node (dots) at (-1.6,-.85) [label=center:$1$]{};
\node (dots) at (-1.5,1) [label=center:$1$]{};
\node (dots) at (-.6,0) [label=center:$2$]{};
\node (dots) at (0,-2.7) [label=center:{\bf (b)}]{};
\node[place] (v_0) at (4,0) {};      
\draw[blue,thick] (v_0) -- +(0:1.5cm);
\draw[blue,thick] (v_0) -- +(30:1.5cm);
\draw[blue,thick] (v_0) -- +(60:1.5cm);
\draw[blue,thick] (v_0) -- +(90:1.5cm);
\draw[blue,thick] (v_0) -- +(210:1.5cm);
\draw[blue,thick] (v_0) -- +(240:1.5cm);
\draw[blue,thick] (v_0) -- +(300:1.5cm);
\draw[blue,thick] (v_0) -- +(330:1.5cm);
\filldraw[red,draw=black] (v_0)  ++(0:1.5cm) circle (2.2pt);
\filldraw[red,draw=black] (v_0)  ++(30:1.5cm) circle (2.2pt);
\filldraw[red,draw=black] (v_0)  ++(60:1.5cm) circle (2.2pt);
\filldraw[red,draw=black] (v_0)  ++(90:1.5cm) circle (2.2pt);
\filldraw[red,draw=black] (v_0)  ++(210:1.5cm) circle (2.2pt);
\filldraw[red,draw=black] (v_0)  ++(240:1.5cm) circle (2.2pt);
\filldraw[red,draw=black] (v_0)  ++(300:1.5cm) circle (2.2pt);
\filldraw[red,draw=black] (v_0)  ++(330:1.5cm) circle (2.2pt);
\filldraw[red,draw=black] (v_0)  ++(360:1.5cm) circle (2.2pt);
\draw[blue,thick] (v_0) circle (1.5cm);
\node (dots) at (4,1.8) [label=center:$-1$]{};
\node (dots) at (4.9,1.6) [label=center:$-1$]{};
\node (dots) at (5.5,1) [label=center:$2$]{};
\node (dots) at (5.9,0) [label=center:$-1$]{};
\node (dots) at (5.6,-.9) [label=center:$-1$]{};
\node (dots) at (4.9,-1.6) [label=center:$1$]{};
\node (dots) at (2.9,-1.5) [label=center:$-1$]{};
\node (dots) at (2.4,-.85) [label=center:$1$]{};
\node (dots) at (3.6,0.2) [label=center:$2$]{};
\node (dots) at (4,-2.7) [label=center:{\bf (c)}]{};
\end{tikzpicture}

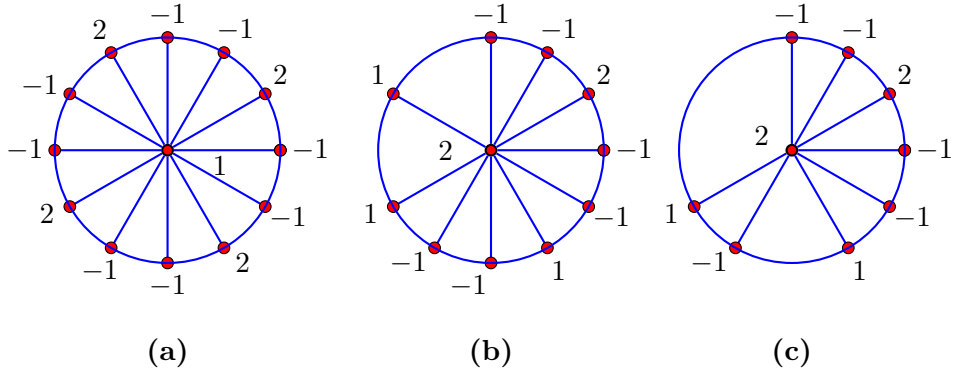
\captionof{figure}{Signed Roman domination labeling of $W_{12}$, $W_{10}$ and $W_8$.\label{fig:w12}}
\end{center}
\begin{center}
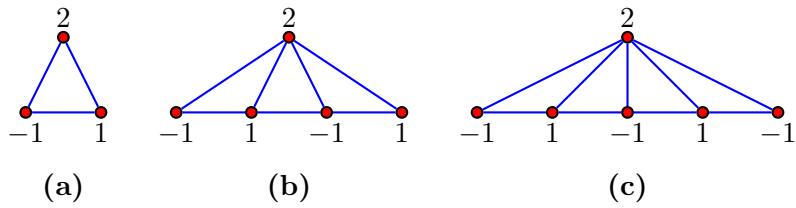

\begin{tikzpicture}
[inner sep=0.5mm, place/.style={circle,draw=black,fill=red,thick}]
\node[place] (v_1) at (-5,0) [label=below:$-1$] {};
\node[place] (v_2) at (-4,0) [label=below:$1$] {}edge [-,thick,blue](v_1);
\node[place] (v_0) at (-4.5,1) [label=above:$2$] {}edge [-,thick,blue](v_1) edge [-,thick,blue](v_2);
\node (dots) at (-4.5,-1) [label=center:{\bf (a)}]{};
\node[place] (v_1) at (-3,0) [label=below:$-1$] {};
\node[place] (v_2) at (-2,0) [label=below:$1$] {}edge [-,thick,blue](v_1);
\node[place] (v_3) at (-1,0) [label=below:$-1$] {}edge [-,thick,blue](v_2);
\node[place] (v_4) at (0,0) [label=below:$1$] {}edge [-,thick,blue](v_3);
\node[place] (v_0) at (-1.5,1) [label=above:$2$] {}edge [-,thick,blue](v_1) edge [-,thick,blue](v_2) edge [-,thick,blue](v_3) edge [-,thick,blue](v_4);
\node (dots) at (-1.5,-1) [label=center:{\bf (b)}]{};
\node[place] (v_1) at (1,0) [label=below:$-1$] {};
\node[place] (v_2) at (2,0) [label=below:$1$] {}edge [-,thick,blue](v_1);
\node[place] (v_3) at (3,0) [label=below:$-1$] {}edge [-,thick,blue](v_2);
\node[place] (v_4) at (4,0) [label=below:$1$] {}edge [-,thick,blue](v_3);
\node[place] (v_5) at (5,0) [label=below:$-1$] {}edge [-,thick,blue](v_4);
\node[place] (v_0) at (3,1) [label=above:$2$] {}edge [-,thick,blue](v_1) edge [-,thick,blue](v_2) edge [-,thick,blue](v_3) edge [-,thick,blue](v_4) edge [-,thick,blue](v_5);
\node (dots) at (3,-1) [label=center:{\bf (c)}]{};
\end{tikzpicture}
\captionof{figure}{Signed Roman domination labeling on $F_2$, $F_4$ and $F_5$, respectively.\label{fig:F2F4F5}}
\end{center}
\begin{center}
\begin{tikzpicture}
[inner sep=0.5mm, place/.style={circle,draw=black,fill=red,thick}]
\node[place] (v_0) at (-4.3,0) {};      
\draw[blue,thick] (v_0) -- +(0:1.5cm);
\draw[blue,thick] (v_0) -- +(30:1.5cm);
\draw[blue,thick] (v_0) -- +(60:1.5cm);
\draw[blue,thick] (v_0) -- +(90:1.5cm);
\draw[blue,thick] (v_0) -- +(120:1.5cm);
\draw[blue,thick] (v_0) -- +(150:1.5cm);
\draw[blue,thick] (v_0) -- +(180:1.5cm);
\draw[blue,thick] (v_0) -- +(210:1.5cm);
\draw[blue,thick] (v_0) -- +(240:1.5cm);
\draw[blue,thick] (v_0) -- +(270:1.5cm);
\draw[blue,thick] (v_0) -- +(300:1.5cm);
\draw[blue,thick] (v_0) -- +(330:1.5cm);
\filldraw[red,draw=black] (v_0)  ++(0:1.5cm) circle (2.2pt);
\filldraw[red,draw=black] (v_0)  ++(30:1.5cm) circle (2.2pt);
\filldraw[red,draw=black] (v_0)  ++(60:1.5cm) circle (2.2pt);
\filldraw[red,draw=black] (v_0)  ++(90:1.5cm) circle (2.2pt);
\filldraw[red,draw=black] (v_0)  ++(120:1.5cm) circle (2.2pt);
\filldraw[red,draw=black] (v_0)  ++(150:1.5cm) circle (2.2pt);
\filldraw[red,draw=black] (v_0)  ++(180:1.5cm) circle (2.2pt);
\filldraw[red,draw=black] (v_0)  ++(210:1.5cm) circle (2.2pt);
\filldraw[red,draw=black] (v_0)  ++(240:1.5cm) circle (2.2pt);
\filldraw[red,draw=black] (v_0)  ++(270:1.5cm) circle (2.2pt);
\filldraw[red,draw=black] (v_0)  ++(300:1.5cm) circle (2.2pt);
\filldraw[red,draw=black] (v_0)  ++(330:1.5cm) circle (2.2pt);
\filldraw[red,draw=black] (v_0)  ++(360:1.5cm) circle (2.2pt);
\draw[blue,thick] (v_0)  ++(120:1.5cm) arc (120:450:1.5cm);
\node (dots) at (-4,1.8) [label=center:$-1$]{};
\node (dots) at (-3.4,1.6) [label=center:$-1$]{};
\node (dots) at (-2.8,1) [label=center:$2$]{};
\node (dots) at (-2.4,0) [label=center:$-1$]{};
\node (dots) at (-2.7,-.9) [label=center:$-1$]{};
\node (dots) at (-3.3,-1.5) [label=center:$2$]{};
\node (dots) at (-4.3,-1.8) [label=center:$-1$]{};
\node (dots) at (-5.2,-1.6) [label=center:$-1$]{};
\node (dots) at (-5.9,-.85) [label=center:$2$]{};
\node (dots) at (-6.2,0) [label=center:$-1$]{};
\node (dots) at (-6,0.85) [label=center:$-1$]{};
\node (dots) at (-5.2,1.6) [label=center:$2$]{};
\node (dots) at (-3.6,-0.2) [label=center:$1$]{};
\node (dots) at (-4.3,-2.7) [label=center:{\bf (a)}]{};
\node[place] (v_0) at (0,0) {};      
\draw[blue,thick] (v_0) -- +(0:1.5cm);
\draw[blue,thick] (v_0) -- +(30:1.5cm);
\draw[blue,thick] (v_0) -- +(60:1.5cm);
\draw[blue,thick] (v_0) -- +(90:1.5cm);
\draw[blue,thick] (v_0) -- +(150:1.5cm);
\draw[blue,thick] (v_0) -- +(210:1.5cm);
\draw[blue,thick] (v_0) -- +(240:1.5cm);
\draw[blue,thick] (v_0) -- +(270:1.5cm);
\draw[blue,thick] (v_0) -- +(300:1.5cm);
\draw[blue,thick] (v_0) -- +(330:1.5cm);
\filldraw[red,draw=black] (v_0)  ++(0:1.5cm) circle (2.2pt);
\filldraw[red,draw=black] (v_0)  ++(30:1.5cm) circle (2.2pt);
\filldraw[red,draw=black] (v_0)  ++(60:1.5cm) circle (2.2pt);
\filldraw[red,draw=black] (v_0)  ++(90:1.5cm) circle (2.2pt);
\filldraw[red,draw=black] (v_0)  ++(150:1.5cm) circle (2.2pt);
\filldraw[red,draw=black] (v_0)  ++(210:1.5cm) circle (2.2pt);
\filldraw[red,draw=black] (v_0)  ++(240:1.5cm) circle (2.2pt);
\filldraw[red,draw=black] (v_0)  ++(270:1.5cm) circle (2.2pt);
\filldraw[red,draw=black] (v_0)  ++(300:1.5cm) circle (2.2pt);
\filldraw[red,draw=black] (v_0)  ++(330:1.5cm) circle (2.2pt);
\filldraw[red,draw=black] (v_0)  ++(360:1.5cm) circle (2.2pt);
\draw[blue,thick] (v_0)  ++(150:1.5cm) arc (150:450:1.5cm);
\node (dots) at (0,1.8) [label=center:$-1$]{};
\node (dots) at (.9,1.6) [label=center:$-1$]{};
\node (dots) at (1.5,1) [label=center:$2$]{};
\node (dots) at (1.9,0) [label=center:$-1$]{};
\node (dots) at (1.6,-.9) [label=center:$-1$]{};
\node (dots) at (.9,-1.6) [label=center:$1$]{};
\node (dots) at (-.3,-1.8) [label=center:$-1$]{};
\node (dots) at (-1.1,-1.5) [label=center:$-1$]{};
\node (dots) at (-1.6,-.85) [label=center:$1$]{};
\node (dots) at (-1.5,1) [label=center:$1$]{};
\node (dots) at (-.6,0) [label=center:$2$]{};
\node (dots) at (0,-2.7) [label=center:{\bf (b)}]{};
\node[place] (v_0) at (4,0) {};      
\draw[blue,thick] (v_0) -- +(0:1.5cm);
\draw[blue,thick] (v_0) -- +(30:1.5cm);
\draw[blue,thick] (v_0) -- +(60:1.5cm);
\draw[blue,thick] (v_0) -- +(90:1.5cm);
\draw[blue,thick] (v_0) -- +(210:1.5cm);
\draw[blue,thick] (v_0) -- +(240:1.5cm);
\draw[blue,thick] (v_0) -- +(300:1.5cm);
\draw[blue,thick] (v_0) -- +(330:1.5cm);
\filldraw[red,draw=black] (v_0)  ++(0:1.5cm) circle (2.2pt);
\filldraw[red,draw=black] (v_0)  ++(30:1.5cm) circle (2.2pt);
\filldraw[red,draw=black] (v_0)  ++(60:1.5cm) circle (2.2pt);
\filldraw[red,draw=black] (v_0)  ++(90:1.5cm) circle (2.2pt);
\filldraw[red,draw=black] (v_0)  ++(210:1.5cm) circle (2.2pt);
\filldraw[red,draw=black] (v_0)  ++(240:1.5cm) circle (2.2pt);
\filldraw[red,draw=black] (v_0)  ++(300:1.5cm) circle (2.2pt);
\filldraw[red,draw=black] (v_0)  ++(330:1.5cm) circle (2.2pt);
\filldraw[red,draw=black] (v_0)  ++(360:1.5cm) circle (2.2pt);
\draw[blue,thick] (v_0)  ++(210:1.5cm) arc (210:450:1.5cm);
\node (dots) at (4,1.8) [label=center:$-1$]{};
\node (dots) at (4.9,1.6) [label=center:$-1$]{};
\node (dots) at (5.5,1) [label=center:$2$]{};
\node (dots) at (5.9,0) [label=center:$-1$]{};
\node (dots) at (5.6,-.9) [label=center:$-1$]{};
\node (dots) at (4.9,-1.6) [label=center:$1$]{};
\node (dots) at (2.9,-1.5) [label=center:$-1$]{};
\node (dots) at (2.4,-.85) [label=center:$1$]{};
\node (dots) at (3.6,0.2) [label=center:$2$]{};
\node (dots) at (4,-2.7) [label=center:{\bf (c)}]{};
\end{tikzpicture}

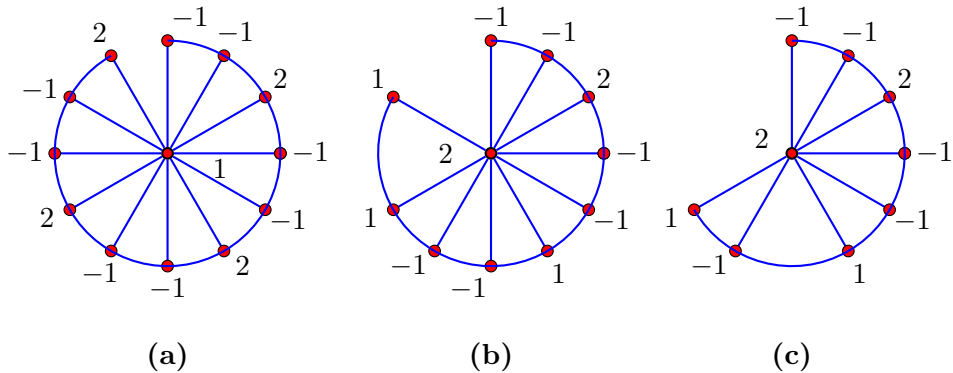
\captionof{figure}{Signed Roman domination labeling on $F_{12}$, $F_{10}$ and $F_8$, respectively.\label{fig:F12F10F8}}
\end{center}
Structures of  $F_n$ and $W_n$ are similar. 
This similarity helps us to provide signed Roman domination functions on $F_n$ using what we construct for $W_n$. 

\begin{thm} \label{theorem:fan}
Let $F_n=K_1\vee P_n$ be a fan of order $n+1$. Then
\begin{eqnarray*}
\gamma_{_{sR}}(F_n)=\left\{ \begin{array}{ll} 2 & n\in\{2,4\} \\ 1 & n\notin\{2,4\}. \end{array}   \right.
\end{eqnarray*}
\end{thm}
\begin{proof}{
Let $V(F_n)=\{v_0,v_1,v_2,...,v_n\}$ and $E(F_n)=\{v_0v_i:~1\leq i\leq n \}\cup \{v_1v_2,v_2v_3,...,v_{n-1}v_n\}$.
Since $\Delta(F_n)=|V(F_n)|-1$, Lemma \ref{MaxDeg} implies that $\gamma_{_{sR}}(F_n)\geq 1$.
$F_2$ is a complete graph with tree vertices and hence $\gamma_{_{sR}}(F_2)=\gamma_{_{sR}}(K_3)=2$. For the case $n=4$ it is not hard to check by inspection that there exists no signed Roman domination function on $F_4$ of weight 1. Figure \ref{fig:F2F4F5} (a) and (b) illustrate a $SRDF$ of weight 2 on $F_2$ and $F_4$, respectively. Thus, for $n\in\{2,4\}$ we have $\gamma_{_{sR}}(F_n)=2$.

To complete the proof it is sufficient to provide a signed Roman domination function of weight 1 on $F_n$ for each $n\notin\{2,4\}$. 
Regards to the different possible cases for $n$ like cases 1 to 4, consider the functions which are defined in the equations \ref{case1}, \ref{case2}, \ref{case3}, and \ref{case4}.
For instance, an optimal SRDF on $F_5$ is depicted in Figure \ref{fig:F2F4F5} (c), where the top vertex is $v_0$ and its below lef one is $v_1$. 
Also, optimal SRDF's on  $F_{12}$, $F_{10}$ and $F_{8}$ are illustrated in Figure \ref{fig:F12F10F8} (a), (b) and (c), respectively (where the central vertex is $v_0$ and the top one is $v_1$).
}\end{proof}

\begin{thm} \label{friendship1}
Let $m\geq2$ be an integer and $n=2m+1$. Then, the signed Roman domination number of the Friendship graph $Fr_n=K_1\vee (mK_2)$ is given by $\gamma_{_{sR}}(Fr_n)=2$.
\end{thm}
\begin{proof}{
Lett $V(Fr_n)=\{x\}\cup\{y_i,z_i:~1\leq i\leq m\}$ and $E(Fr_n)=\{xy_i,xz_i:~1\leq i\leq m\}\cup\{y_iz_i:~1\leq i\leq m\}$.
Since $\Delta(Fr_n)=|V(Fr_n)|-1$, Lemma \ref{MaxDeg} implies that $\gamma_{_{sR}}(Fr_n)\geq1$.
Consider the function $g$ defined from $V(Fr_n)$ to the set $\{-1,1,2\}$  as follows.
\begin{eqnarray*}
g(v)=\left\{ \begin{array}{ll}  2 & v=x\\ 1 & v\in\{y_1,y_2,...,y_m\} \\ -1 & v\in\{z_1,z_2,...,z_m\}. \end{array}  \right.
\end{eqnarray*}
Since $g$ is a $SRDF$ on $Fr_n$, we get $\gamma_{_{sR}}(Fr_n)\leq2$.
Now let $f=(V_{-1},V_1,V_2)$ be a optimal signed Roman domination function on $Fr_n$.
If $V_{-1}=\emptyset$, then $w(f)\geq n\geq 5$, a contradiction.
Hence $|V_{-1}|\geq 1$ and this implies that $|V_2|\geq 1$.
If $f(y_i)=f(z_i)=-1$ for some $i$, then $f(N_{Fr_n}[y_i])\leq0$, which is a contradiction.
Thus, for each $i\in \{1,2,...,m\}$ we have $|V_{-1}\cap \{y_i,z_i\}|\leq 1$ and this implies that $|V_{-1}|\leq m+1$.
If $|V_{-1}|=m+1$, then $|V_{-1}\cap \{y_i,z_i\}|=1$ for each $i\in \{1,2,...,m\}$, and $x\in V_{-1}$.
Hence,  $f(N_{Fr_n}[y_1])=f(y_1)+f(z_1)+f(x)\leq 0$ which is a contradiction.
Therefore, $|V_{-1}|\leq m$ and 
$$\gamma_{_{sR}}(Fr_n)=w(f)=2|V_2|+|V_1|-|V_{-1}|\geq 2\times 1+m\times 1 + m\times (-1)=2.$$
This completes the proof.
}\end{proof}


\begin{thebibliography}{20}

\bibitem{Ahangar} H. A. Ahangar, M. A. Henning, Y. Zhao,  C. L\"{o}wenstein, V. Samodivkin,
   Signed Roman domination in graphs, {\em J. Comb. Optim.},  {\bf 27} (2014) 241-255.

\bibitem{Berge} C. Berge, Graphs and hypergraphs, {\em North Holland}, Amsterdam, (1973).

\bibitem{Furedi} Z. F\"{u}redi and D. Mubayi,  Signed domination in regular graphs and setsystems, {\em J. Combin. Theory Ser. B} {\bf 76} (1999) 223-239.

\bibitem{complexity} M. R. Garey and D. S. Johnson, Computers and Intractability: A Guide to the
theory of NP-completeness, {\em W.H. Freeman}, San Francisco (1979).

\bibitem{Hynes} T. W. Haynes, S. T. Hedetniemi and P. J. Slater, Domination in Graphs, {\em Advanced
Topics, Marcel Dekker}, New York, (1998).

\bibitem{Hening} M. A. Henning and S. T. Hedetniemi, Defending the Roman empirea new strategy, {\em Discrete Math.}, {\bf 266}, (2003), 239251

\bibitem{Ore}  O. Ore,  Theory of graphs, {\em Amer. Math. Soc. Colloq. Publ.}, {\bf 38}, Providence,
(1962).

\bibitem{domatic}  S. M. Sheikholeslami, L. Volkmann,  The signed Roman domatic number of a graph,
{\em Annales Mathematicae et Informaticae}, {\bf 40} (2012) 105112.

\bibitem{Stewart}  I. Stewart,  Defend the Roman Empire, {\em Sci. Amer.}, {\bf 281} (1999) 136-139.

\bibitem{west}  D. B. West, Introduction to graph theory,
 {\em Prentice Hall Inc.}, Upper Saddle River, NJ 07458, Second
Edition (2001).

\end{thebibliography}
\end{document}